\documentclass[10 pt]{amsart}
\usepackage{amsfonts}
\usepackage{ifthen}
\usepackage{amsthm}
\usepackage{amsmath}
\usepackage{graphicx}
\usepackage{amscd,amssymb,amsthm}
\usepackage{graphicx}
\usepackage{epstopdf}
\usepackage{hyperref}
\usepackage{cite}
\usepackage{filecontents}
\usepackage[numbers]{natbib}

\newcounter{minutes}
\divide\time by 60
\newcounter{hours}
\multiply\time by 60 \addtocounter{minutes}{-\time}

\newtheorem{definition}{Definition}

\newtheorem{lemma}{Lemma}
\newtheorem{theorem}{Theorem}
\newtheorem{corollary}{Corollary}
\newtheorem{remark}{Remark}

\newcommand{\real}{\operatorname{Re}}


\begin{document}
\markboth{Halit Orhan, Evr{\.i}m Toklu and Ekrem Kad{\i}o{\u{g}}lu}{On the Chebyshev polynomials and second Hankel Determinant}
\title{Second Hankel determinant for certain subclasses of bi-univalent functions involving Chebyshev polynomials}
\title[Second Hankel determinant for...]{Second Hankel determinant for certain subclasses of bi-univalent functions involving Chebyshev polynomials}

\author[H. Orhan]{Hal{\.I}t Orhan}
\address{Department of Mathematics, Faculty of Science, Atat\"urk University, 25240 Erzurum, Turkey} \email{horhan@atauni.edu.tr}

\author[E. Toklu]{Evr{\.I}m Toklu}
\address{Department of Mathematics, Faculty of Science, A\u{g}r{\i} {\.I}brah{\.I}m \c{C}e\c{c}en University, 04100 A\u{g}r{\i}, Turkey} \email{etoklu@agri.edu.tr}

\author[E. Kad{\i}o{\u{g}}lu]{Ekrem Kad{\i}o{\u{g}}lu}
\address{Department of Mathematics, Faculty of Science, Atat\"urk University, 25240 Erzurum, Turkey} \email{ekrem@atauni.edu.tr}

\def\thefootnote{}
\footnotetext{ \texttt{File:~\jobname .tex,
		printed: \number\year-\number\month-\number\day,
		\thehours.\ifnum\theminutes<10{0}\fi\theminutes}
} \makeatletter\def\thefootnote{\@arabic\c@footnote}\makeatother

\maketitle
\begin{abstract}
In this paper our purpose is to find upper bound estimate for the second Hankel determinant $|a_{2}a_{4}-a_{3}^{2}|$ for functions defined by convolution belonging to the class $\mathcal{N}_{\sigma}^{\mu,\delta}(\lambda,t)$ by using Chebyshev polynomials.
\end{abstract}

\section{\bf Introduction}
Let $\mathcal{A}$ denote the class of functions $f(z)$ analytic in the open unit disk  $\mathbb{U}:=\left\lbrace z\in \mathbb{C}:|z|<1\right\rbrace $ and normalized by
\begin{equation}\label{equ1}
f(z)=z+\sum_{n\geq2}a_{n}z^{n}.
\end{equation} 

Because of the Koebe one-quarter theorem it is well known that every univalent function $f\in \mathcal{A}$ has an inverse $f^{-1}:f(\mathbb{U})\rightarrow \mathbb{U}$ satisfying
\begin{equation*}
f^{-1}\left( f(z)\right) =z, \text{ \  \ } (z\in \mathbb{U})
\end{equation*}
and
\begin{equation*}
f\left( f^{-1}(w)\right) =w,\text{ \ \ } (|w|<1/4).
\end{equation*}
Moreover, it is easy to check that the inverse function has the series expansion of the form
\begin{equation}\label{equ2}
f^{-1}(w)=w-a_{2}w^{2}+(2a_{2}^{2}-a_{3})w^{3}-(5a_{2}^{3}-5a_{2}a_{3}+a_{4})w^{4}+..., \text{ \  \ } w\in f(\mathbb{U}).
\end{equation}
A function $f\in \mathcal{A}$ is said to be bi-univalent in $\mathbb{U}$ if both $f$ and its inverse $g=f^{-1}$ are univalent in $\mathbb{U}$. Let $\sigma$ denote the class of bi-univalent functions in $\mathbb{U}$ given by (\ref{equ1}). For a brief history of functions in the class $\sigma$, and also various other properties of the bi-univalent function one can see recent works \cite{alr, coy, dco, omb, smg} and references therein.

Some of the prominent and well-examined subclasses of univalent functions class $\mathcal{S}$ are the class $\mathcal{S}^{\star}(\alpha)$ of starlike functions of order $\alpha$ in $\mathbb{U}$ and the class $\mathcal{K}(\alpha)$ of convex functions of order $\alpha$ in $\mathbb{U}$. Moreover, by means of the analytic descriptions, we have
$$\mathcal{S}^{\star}(\alpha):=\left\lbrace f: f\in \mathcal{A} \text{ \ and  \ } \real\left( \frac{zf^{\prime}(z)}{f(z)}\right)>\alpha;\text{ \ \ } z\in \mathbb{U};\text{ \ \ } 0\leq \alpha<1  \right\rbrace $$
and $$\mathcal{K}(\alpha):=\left\lbrace f: f\in \mathcal{A} \text{ \ and  \ } \real\left( 1+\frac{zf^{\prime\prime}(z)}{f^{\prime}(z)}\right)>\alpha;\text{ \ \ } z\in \mathbb{U};\text{ \ \ } 0\leq \alpha<1\right\rbrace .$$
For $0\leq \alpha<1$, a function $f\in \sigma$ is in the class $\mathcal{S}_{\sigma}^{\star}(\alpha)$ of bi-starlike function of order $\alpha$, or $\mathcal{K}_{\sigma}(\alpha)$ of bi-convex function of order $\alpha$ if both $f$ and its inverse $f^{-1}$ are, respectively, starlike or convex functions of order $\alpha$.

We say that $f\in\mathcal{A}$  is subordinate to the function  $g\in\mathcal{A}$ in  $\mathbb{U}$, written $f(z)\prec g(z)$, if there exists a Schwarz function $w$, analytic in  $\mathbb{U}$, with $w(0)=0$ and $|w(z)|<1$, and such that $f(z)=g(w(z))$.

For $f(z)$ given by (\ref{equ1}) and $\Theta(z)$ defined by
\begin{equation}\label{functionTheta}
\Theta(z)=z+\sum_{n\geq2}\theta_{n}z^{n}, \text{ \ \ } (\theta_{n}\geq0),
\end{equation}
the Hadamard product (or convolution) $(f*\Theta)(z)$ of the functions $f(z)$ and $\Theta(z)$ defined by
\begin{align}\label{Hadamard}
(f*\Theta)(z)=z+\sum_{n\geq2}a_{n}\theta_{n}z^{n}=(\Theta*f)(z).
\end{align}

Next, we consider the function
\begin{align}
f_{\delta}(z)&=\int_{0}^{z}(\frac{1+r}{1-r})^{\delta}\frac{1}{1-r^{2}}dr\nonumber\\&= z+\delta z^{2}+\frac{1}{3}(2\delta^{2}+1)z^{3}+... \label{functionnew} \\ \nonumber &=z+\sum_{n\geq2}b_{n}(\delta)z^{n}, \text{ \ \ }(\delta>0,\text{ \ \ } z\in \mathbb{U}). 
\end{align}
It is worth mentioning that if $\delta<1$, then $zf^{\prime}_{\delta}(z)$ is starlike with two slits. Moreover, we can see that since $zf^{\prime}_{1}(z)$ is the Koebe function, all the functions $f_{\delta}$ are univalent and convex in $\mathbb{U}$. For more detail about the function $f_{\delta}(z)$ one can refer to \cite{trimble}. If we put the function $f_{\delta}(z)$ defined by (\ref{functionnew}) in for the function $\Theta(z)$ given by (\ref{functionTheta}) in the equality (\ref{Hadamard}), we have
\begin{equation}\label{functionh}
h_{\delta}(z)=(f*f_{\delta})(z)=z+\sum_{n\geq2}a_{n}b_{n}(\delta)z^{n}=(f_{\delta}*f)(z).
\end{equation}

In 1976, Noonan and Thomas \cite{nt} defined the $q^{th}$ Hankel determinant of $f$ given by (\ref{equ1}) for integers $n\geq1$ and $q\geq1$ by

$$H_{q}(n)= \left| 
\begin{matrix}
a_{n} & a_{n+1} & \ldots &a_{n+q-1}\\
a_{n+1} & a_{n+2} & \ldots & a_{n+q-2}\\
\vdots &\vdots&\vdots&\vdots\\
a_{n+q-1} & a_{n+q-2} & \ldots & a_{n+2q-2}	
\end{matrix}
\right|, \text{ \ \ } (a_{1}=1). $$
This determinant has been investigated by several authors in the literature \cite{dienes,nt}. For instance, this determinant is so useful in showing that a function of bounded characteristic in $\mathbb{U}$, i.e. a function that is a ratio of two bounded analytic functions with its Laurent series around the origin having integral coefficients, is rational \cite{cantor}. Moreover, it is important to mention that the Hankel detrminants $H_{2}(1)=a_{3}-a_{2}^{2}$ and $H_{2}(2)=a_{2}a_{4}-a_{3}^{2}$ are well-known as Fekete-Szeg\"{o} and second Hankel determinant functionals, respectively. In 1969, the Fekete-Szeg\"{o} problem for the classes $\mathcal{S}^{\star}$ and $\mathcal{K}$ was investigated by Keogh and Merkes \cite{keme}. Recently, many authors have discussed upper bounds for the Hankel determinant of functions belonging to various subclasses of univalent functions \cite{alrs,dt,lrs} and references therein. Very recently, the upper bounds of $H_{2}(2)$ for the classes  $\mathcal{S}_{\sigma}^{\star}(\alpha)$ and $\mathcal{K}_{\sigma}(\alpha)$ were investigated by Deniz at al. \cite{dco}. Latter, the works were extended by Orhan et al.\cite{omy, otk} and Alt{\i}nkaya and Yal\c{c}{\i}n \cite{ay1, ay2}.

Chebyshev polynomials, which are used by us in this study, play an important role in many branches of mathematics, especially in numerical analysis (see \cite{chihara}). We know that there are several kinds of Chebyshev polynomials. In particular we shall introduce the first and second kind polynomials  $T_{n}(x)$ and  $U_{n}(x)$. For a brief history of the Chebyshev polynomials of first kind $T_{n}(x)$ and second kind $U_{n}(x)$  and their numerous uses in different applications one can refer \cite{doha,drs,ay0}.

The most remarkable kinds of the Chebyshev polynomials are the first and second kinds and in the case of real variable $x$ on $(-1,1)$ they are defined by
$$T_{n}(x)=\cos(n\arccos x) \text{ \and \ } U_{n}(x)=\frac{\sin[(n+1)\arccos x]}{\sin(\arccos x)}=\frac{\sin[(n+1)\arccos x ]}{\sqrt{1-x^2}}.$$
Now, we consider the function which is the generating function of a Chebyshev polynomial
$$G(t,z)=\frac{1}{1-2tz+z^2}, \text{ \ \ } t\in(\frac{1}{2},1), z\in \mathbb{U}.$$
It is well-known that if $t=\cos\theta, \text{ \ \ } t=(-\pi/3,\pi/3)$, then
\begin{align*}
G(t,z)&=1+\sum_{n\geq1}\frac{\sin(n+1)\theta}{\sin\theta}z^{n}\\&=1+2\cos\theta z+(3\cos^{2}\theta-\sin^{2}\theta)z^{2}+..., \text{ \ \ } (z\in\mathbb{U}).
\end{align*}
That is, in view of \cite{ww} we can write
\begin{equation}\label{equ3}
G(t,z)=1+U_{1}(t)z+U_{2}(t)z^2+U_{3}(t)z^3+..., \text{ \ \ } t\in(\frac{1}{2},1),z\in \mathbb{U},
\end{equation}
where, $U_{n}(t)$ stands for the second kind Chebyshev polynomials. From the definition of the second kind Chebyshev polynomials, we easily arrive at $U_{1}(t)=2t$. Also, it is well-known that we have the following reccurence relation
$$U_{n+1}(t)=2tU_{n}(t)-U_{n-2}(t)$$
for all $n\in \mathbb{N}$. From here, we can easily obtain
\begin{equation}
U_{1}(t)=2t,\text{ \ \ } U_{2}(t)=4t^{2}-1, \text{ \ \ } U_{3}(t)=8t^{3}-4t, \text{ \ \ } U_{4}(t)=16t^{4}-12t^{2}+1,..._{.}
\end{equation}
\begin{definition}
For $\lambda\geq1$, $\mu\geq0$, $\delta\geq1$ and $t\in(1/2,1]$, a function $h_{\delta}\in \sigma$ given by (\ref{functionh}) is said to be in class $\mathcal{N}_{\sigma}^{\mu, \delta}(\lambda, t)$ if the following subordinations hold for all $z,w\in \mathbb{U}$:
\begin{equation}\label{equ5}
(1-\lambda)(\frac{h_{\delta}(z)}{z})^{\mu}+\lambda h_{\delta}^{\prime}(z)(\frac{h_{\delta}(z)}{z})^{\mu-1}\prec G(z,t)
\end{equation}
and
\begin{equation}\label{equ6}
(1-\lambda)(\frac{k_{\delta}(w)}{w})^{\mu}+\lambda k_{\delta}^{\prime}(w)(\frac{k_{\delta}(w)}{w})^{\mu-1}\prec G(w,t),
\end{equation}
where the function $k_{\delta}=h_{\delta}^{-1}$ is defined by (\ref{equ2}).
\end{definition}

Obviously, for $\delta=1$, we get that $\mathcal{N}_{\sigma}^{\mu,1}(\lambda,t)=\mathcal{N}_{\sigma}^{\mu}(\lambda,t)$. It is important to mention that the class $\mathcal{N}_{\sigma}^{\mu}(\lambda,t)$ was introduced and investigated by Bulut et al. \cite{bma}. Also, they discussed initial coefficient estimates and Fkete-Szeg\"{o} bounds for the class $\mathcal{N}_{\sigma}^{\mu}(\lambda,t)$ and its subclasses given in the following remark.
\begin{remark}
\begin{enumerate}
\item[(i)] For $\delta=1$ and $\mu=1$, we get the class $\mathcal{N}_{\sigma}^{1,1}(\lambda,t)=\mathcal{B}_{\sigma}(\lambda,t)$ consist of functions $f\in \sigma$ satisfying the condition
$$(1-\lambda)\frac{f(z)}{z}+\lambda f^{\prime}(z)\prec G(z,t)$$
and
$$(1-\lambda)\frac{g(w)}{w}+\lambda g^{\prime}(w)\prec G(w,t)$$
where the function $g=f^{-1}$ is defined by (\ref{equ2}). This class was introduced and studied by Bulut et al \cite{bmb} (see also \cite{mustafa}).
\item[(ii)] For $\delta=1$ and $\lambda=1$, we obtain the class $\mathcal{N}_{\sigma}^{\mu,1}(1,t)=\mathcal{B}_{\sigma}^{\mu}(t)$ consist of bi-Bazilevi\u{c} functions:
$$f^{\prime}(z)\left(\frac{f(z)}{z} \right)^{\mu-1}\prec G(z,t) $$
and
$$g^{\prime}(w)\left(\frac{g(w)}{w} \right)^{\mu-1}\prec G(w,t),$$
where the function $g=f^{-1}$ is defined by (\ref{equ2}). This class was introduced and studied by Alt{\i}nkaya and Yal\c{c}{\i}n \cite{ay3}.
\item[(iii)] For $\delta=1$, $\mu=1$ and $\lambda=1$, we have the class $\mathcal{N}_{\sigma}^{1,1}(1,t)=\mathcal{B}_{\sigma}(t)$ consist of functions $f$ satisfying the condition
$$f^{\prime}(z)\prec G(z,t)$$
and
$$g^{\prime}(w)\prec G(w,t)$$
where the function $g=f^{-1}$ is defined by (\ref{equ2}).
\item[(iv)] For $\delta=1$, $\lambda=1$ and $\mu=0$, we have the class $\mathcal{N}_{\sigma}^{0,1}=\mathcal{S}_{\sigma}^{\star}(t)$ satisfying the condition
$$\frac{zf^{\prime}(z)}{f(z)}\prec G(z,t)$$
and
$$\frac{wg^{\prime}(z)}{g(w)}\prec G(w,t)$$
where the function $g=f^{-1}$ is defined by (\ref{equ2}).
\end{enumerate}
\end{remark}  

Let us take a look some lemmas which are very useful in building our main results.

Let $\mathcal{P}$ denote the class of analytic functions $p$ in $\mathbb{U}$ such that $p(0)=1 \text{ \ and \ } \real(p(z))>0$, $z\in\mathbb{U}$. Also, we know that this class is usually called the Carath\'{e}odory class.
\begin{lemma} (see \cite{pommerenke})\label{lem1}
If the function $p\in \mathcal{P}$ is given by the following series:
\begin{equation}\label{equ4}
p(z)=1+c_{1}z+c_{2}z^{2}+c_{3}z^{3}+...,
\end{equation}
then the sharp estimate given by
\begin{equation}\label{lem1ineq}
|c_{n}|\leq2 \text{ \ \ } (n=1,2,3,...)
\end{equation}
holds true.
\end{lemma}
\begin{lemma}\cite{gs}\label{lem2}
If the function $p\in \mathcal{P}$ is given by the series (\ref{equ4}), then
\begin{align*}
2c_{2}&=c_{1}^{2}+x(4-c_{1}^{2}),\\4c_{3}&=c_{1}^{3}+2(4-c_{1}^{2})c_{1}x-c_{1}(4-c_{1}^{2})x^{2}+2(4-c_{1}^{2})(1-|x|^{2})z
\end{align*}
for some $x$ and $z$ with $|x|\leq1$ and  $|y|\leq1$. 
\end{lemma}

In the present investigation, we seek upper bound for the second Hankel determinant for functions  $h_{\delta}$ belongs to the class  $\mathcal{N}_{\sigma}^{\mu, \delta}(\lambda, t)$ by making use of the Chebyshev polynomials expansions and the Hadamard product. Also, we give some remarkable consequences related to the class $\mathcal{N}_{\sigma}^{\mu, \delta}(\lambda, t)$. 
\section{\bf Main Results}
\setcounter{equation}{0}
\begin{theorem}\label{theo1}
Let $h_{\delta}\in\sigma$ of the form (\ref{functionh}) be in $\mathcal{N}_{\sigma}^{\mu, \delta}(\lambda;t)$. Then
\begin{align*}
\left|a_{2}a_{4}-a_{3} ^{2}\right| &\leq
\begin{cases}
\varphi(2^{-},t)      ,&\chi_{1}\geq0 \text{ \and\ } \chi_{2}\geq0\\
\frac{36t^{2}}{(2\delta^{2}+1)^{2}(2\lambda+\mu)^{2}} ,& \chi_{1}\leq0 \text{ \and\ } \chi_{2}\leq0\\
\max\left\lbrace\frac{36t^{2}}{(2\delta^{2}+1)^{2}(2\lambda+\mu)^{2}},\text{ \ \ }\varphi(2^{-},t) \right\rbrace ,&\chi_{1}>0 \text{ \and\ } \chi_{2}<0\\
\max\left\lbrace\varphi(c_{0},t), \varphi(2^{-},t) \right\rbrace  ,&\chi_{1}<0 \text{ \and\ } \chi_{2}>0
\end{cases},
\end{align*}
where
$$\varphi(2^{-},t)=\frac{9U_{1}^{2}(t)}{(2\delta^{2}+1)^{2}(2\lambda+\mu)^{2}}+\frac{\chi_{1}+9\chi_{2}}{6\delta(\delta^{3}+2\delta)(2\delta^{2}+1)^{2}(3\lambda+\mu)(2\lambda+\mu)^{2}(\lambda+\mu)^{4}},$$
$$\varphi(c_{0},t)=\frac{9U_{1}^{2}(t)}{(2\delta^{2}+1)^{2}(2\lambda+\mu)^{2}}-\frac{27\chi_{2}^{2}}{8\chi_{1}\delta(\delta^{3}+2\delta)(2\delta^{2}+1)^{2}(3\lambda+\mu)(2\lambda+\mu)^{2}(\lambda+\mu)^{4}}, \text{ \ \ \ } c_{0}=\sqrt{\frac{-18\chi_{2}}{\chi_{1}}}$$
and
\begin{align*}
\chi_{1}&=(2\lambda+\mu)^{2}U_{1}(t)\left|\Omega_{\lambda,\mu,\delta}(t) \right|+18(\lambda+\mu)^{3}\bigg( 3\delta(\delta^{3}+2\delta)(\lambda+\mu)(3\lambda+\mu)-(2\delta^{2}+1)^{2}(2\lambda+\mu)^{2}\bigg)U_{1}^{2}(t)
\\&-9(\lambda+\mu)^{2}(2\lambda+\mu)U_{1}(t)\bigg((3\lambda+\mu)(8\delta^{4}-4\delta^{2}+5)U_{1}^{2}(t)+4(2\delta^{2}+1)^{2}(2\lambda+\mu)(\lambda+\mu)U_{2}(t)\bigg),
\\\chi_{2}&=\bigg[(2\lambda+\mu)(3\lambda+\mu)\left(8\delta^{4}-4\delta^{2}+5\right)U_{1}^{3}(t)+4(2\delta^{2}+1)^{2}(\lambda+\mu)(2\lambda+\mu)^{2}U_{1}(t)U_{2}(t)
\\&+(\lambda+\mu)U_{1}^{2}(t)\big(2(2\delta^{2}+1)^{2}(2\lambda+\mu)^{2}-12\delta(\delta^{3}+2\delta)(\lambda+\mu)(3\lambda+\mu)\big) \bigg](\lambda+\mu)^{2} , 
\end{align*}
where, $\Omega_{\lambda,\mu,\delta}(t)=18(2\delta^{2}+1)^{2}(\lambda+\mu)^{3}U_{3}(t)-U_{1}^{3}(t)(3\lambda+\mu)\bigg(3(2\delta^{2}+1)^{2}(\mu^{2}+3\mu-4)+54\delta(\delta^{3}+2\delta)\bigg)$.
\end{theorem}
\begin{proof}
Let $h_{\delta}\in \mathcal{N}_{\sigma}^{\mu,\delta}(\lambda,t)$. Then, we have
\begin{equation}\label{equ8}
(1-\lambda)(\frac{h_{\delta}(z)}{z})^{\mu}+\lambda h_{\delta}^{\prime}(z)(\frac{h_{\delta}(z)}{z})^{\mu-1}=G(t,u(z))
\end{equation}
and
\begin{equation}\label{equ9}
(1-\lambda)(\frac{k_{\delta}(w)}{w})^{\mu}+\lambda k_{\delta}^{\prime}(w)(\frac{k_{\delta}(w)}{w})^{\mu-1}=G(t,v(w))
\end{equation}
where $p_{1},p_{2}\in \mathcal{P}$ and defined by
\begin{equation}\label{equ10}
p_{1}(z)=\frac{1+u(z)}{1-u(z)}=1+c_{1}z+c_{2}z^{2}+c_{3}z^{3}+...
\end{equation}
and
\begin{equation}\label{equ11}
p_{2}(w)=\frac{1+v(w)}{1-v(w)}=1+d_{1}w+d_{2}w^{2}+d_{3}w^{3}+..._{.}
\end{equation}
It follows from (\ref{equ10}) and (\ref{equ11}) that
\begin{equation}\label{equ12}
u(z)=\frac{p_{1}(z)-1}{p_{1}(z)+1}=\frac{1}{2}\left[c_{1}z+\left(c_{2}-\frac{c_{1}^{2}}{2}\right)z^{2}+\left(c_{3}-c_{1}c_{2}+\frac{c_{1}^{3}}{4} \right)z^{3}+...   \right] 
\end{equation}
and
\begin{equation}\label{equ13}
v(w)=\frac{p_{2}(w)-1}{p_{2}(w)+1}=\frac{1}{2}\left[d_{1}w+\left(d_{2}-\frac{d_{1}^{2}}{2}\right)w^{2}+\left(d_{3}-d_{1}d_{2}+\frac{d_{1}^{3}}{4} \right)w^{3}+...   \right].
\end{equation}
Using (\ref{equ12}) together with (\ref{equ13}), taking $G(z,t)$ as given in (\ref{equ3}), we get that
\begin{align}\label{equ14}
G(t,u(z))=1&+\frac{U_{1}(t)}{2}c_{1}z+\left[\frac{U_{1}(t)}{2}\left(c_{2}-\frac{c_{1}^{2}}{2} \right)+\frac{U_{2}(t)}{4}c_{1}^{2}  \right]z^{2}
  \\\nonumber
  &+\left[\frac{U_{1}(t)}{2}\left(c_{3}-c_{1}c_{2}+\frac{c_{1}^{3}}{4} \right)+\frac{U_{2}(t)}{2}c_{1}\left( c_{2}-\frac{c_{1}^{2}}{2}\right)+\frac{U_{3}(t)}{8}c_{1}^{3} \right]z^{3}+...
\end{align}
and
\begin{align}\label{equ15}
G(t,v(w))=1&+\frac{U_{1}(t)}{2}d_{1}w+\left[\frac{U_{1}(t)}{2}\left(d_{2}-\frac{d_{1}^{2}}{2} \right)+\frac{U_{2}(t)}{4}d_{1}^{2}  \right]w^{2}\\ \nonumber&+\left[\frac{U_{1}(t)}{2}\left(d_{3}-d_{1}d_{2}+\frac{d_{1}^{3}}{4} \right)+\frac{U_{2}(t)}{2}d_{1}\left( d_{2}-\frac{d_{1}^{2}}{2}\right)+\frac{U_{3}(t)}{8}d_{1}^{3} \right]w^{3}+..._{.}  
\end{align}
By considering (\ref{equ8}), (\ref{equ14}) and (\ref{equ9}), (\ref{equ15}), when some elementary calculations are done, we get that
\begin{align}
(\lambda+\mu)a_{2}b_{2}(\delta)&=\frac{U_{1}(t)}{2}c_{1}\label{eq0},
\\(2\lambda+\mu)\left[a_{3}b_{3}(\delta)+(\mu-1)\frac{a_{2}^{2}b_{2}^{2}(\delta)}{2} \right]\label{eq1} &=\frac{U_{1}(t)}{2}(c_{2}-\frac{c_{1}^{2}}{2})+\frac{U_{2}(t)}{4}c_{1}^{2},
\\(3\lambda+\mu)\left[ a_{4}b_{4}(\delta)+(\mu-1)a_{2}a_{3}b_{2}(\delta)b_{3}(\delta)+(\mu-1)(\mu-2)\frac{a_{2}^{3}b_{2}^{3}(\delta)}{6}\right]&=\frac{U_{1}(t)}{2}\left( c_{3}-c_{1}c_{2}+\frac{c_{1}^{3}}{4}\right) \nonumber\\  &+\frac{U_{2}(t)}{2}c_{1}\left(c_{2}-\frac{c_{1}^{2}}{2} \right)+\frac{U_{3}(t)}{8}c_{1}^{3} \label{eq2}
\end{align}
and
\begin{align}
-(\lambda+\mu)a_{2}b_{2}(\delta)&=\frac{U_{1}(t)}{2}d_{1}\label{eq3},
\\(2\lambda+\mu)\left[(\mu+3)\frac{a_{2}^{2}b_{2}^{2}(\delta)}{2}-a_{3}b_{3}(\delta) \right]\label{eq4} &=\frac{U_{1}(t)}{2}(d_{2}-\frac{d_{1}^{2}}{2})+\frac{U_{2}(t)}{4}d_{1}^{2},
\\(3\lambda+\mu)\left[ (\mu+4)a_{2}a_{3}b_{2}(\delta)b_{3}(\delta)-(\mu+4)(\mu+5)\frac{a_{2}^{3}b_{2}^{3}(\delta)}{6}-a_{4}b_{4}(\delta)\right]&=\frac{U_{1}(t)}{2}\left( d_{3}-d_{1}d_{2}+\frac{d_{1}^{3}}{4}\right) \nonumber\\  &+\frac{U_{2}(t)}{2}d_{1}\left(d_{2}-\frac{d_{1}^{2}}{2} \right)+\frac{U_{3}(t)}{8}d_{1}^{3} \label{eq5}.
\end{align}
Using (\ref{eq0}) along with (\ref{eq3}), we find that
\begin{equation}\label{equ16}
c_{1}=-d_{1}
\end{equation}
and
\begin{equation}\label{equ17}
a_{2}=\frac{U_{1}(t)}{2\delta(\lambda+\mu)}c_{1}.
\end{equation}

Now, from (\ref{eq1}), (\ref{eq4}) and (\ref{equ17}), we obtain that
\begin{equation}\label{equ18}
a_{3}=\frac{3}{2\delta^{2}+1}\left[\frac{U_{1}^{2}(t)}{4(\lambda+\mu)^{2}}c_{1}^{2}+\frac{U_{1}(t)}{4(2\lambda+\mu)}(c_{2}-d_{2}) \right].
\end{equation}

Also, subtracting (\ref{eq5}) from (\ref{eq2}) and using (\ref{equ17}) together with (\ref{equ18}) we get that
\begin{align}
a_{4}=&\frac{3}{\delta^{3}+2\delta}\bigg[ \left(  \frac{U_{1}(t)-2U_{2}(t)+U_{3}(t)}{8(3\lambda+\mu)}-\frac{(\mu^{2}+3\mu-4)U_{1}^{3}(t)}{48(\lambda+\mu)^{3}}\right) c_{1}^{3}+\frac{5U_{1}^{2}(t)}{16(\lambda+\mu)(2\lambda+\mu)}c_{1}(c_{2}-d_{2})\nonumber\\&+\frac{U_{2}(t)-U_{1}(t)}{4(3\lambda+\mu)}c_{1}(c_{2}+d_{2})+\frac{U_{1}(t)}{4(3\lambda+\mu)}(c_{3}-d_{3}) \bigg]. \label{eq6}
\end{align}
Thus, we can easily determine that
\begin{align}
\left| a_{2}a_{4}-a_{3}^2\right|&=\bigg|\frac{U_{1}(t)\Delta_{\lambda,\mu,\delta}(t) }{96\delta(\delta^{3}+2\delta)(2\delta^{2}+1)^{2}(3\lambda+\mu)(\lambda+\mu)^{4}}c_{1}^{4} +\frac{U_{1}^{3}(t)\left[15(2\delta^{2}+1)^{2}-36\delta(\delta^{3}+2\delta) \right] }{32\delta(\delta^{3}+2\delta)(2\delta^{2}+1)^{2}(\lambda+\mu)^{2}(2\lambda+\mu)}c_{1}^{2}(c_{2}-d_{2})\nonumber\\&+\frac{3U_{1}(t)\left[U_{2}(t)-U_{1}(t) \right] }{8\delta(\delta^{3}+2\delta)(\lambda+\mu)(3\lambda+\mu)}c_{1}^{2}(c_{2}+d_{2}) +\frac{3U_{1}^{2}(t)}{8\delta(\delta^{3}+2\delta)(\lambda+\mu)(3\lambda+\mu)}c_{1}(c_{3}-d_{3})\nonumber\\&-\frac{9U_{1}^{2}(t)}{16(2\delta^{2}+1)^{2}(2\lambda+\mu)^{2}}(c_{2}-d_{2})^{2}\bigg|,\label{eq7}
\end{align}
where $$\Delta_{\lambda,\mu,\delta}(t)=18(2\delta^{2}+1)^{2}(\lambda+\mu)^{3}((U_{1}(t)-2U_{2}(t)+U_{3}(t))-U_{1}^{3}(t)(3\lambda+\mu)\big(3(2\delta^{2}+1)^{2}(\mu^{2}+3\mu-4)+54\delta(\delta^{3}+2\delta)\big).$$
In view of Lemma \ref{lem2} and (\ref{equ16}), we write
\begin{align}
c_{2}-d_{2}&=\frac{4-c_{1}^{2}}{2}(x-y)\label{eq8},\\c_{2}+d_{2}&=c_{1}^{2}+\frac{4-c_{1}^{2}}{2}(x+y)\label{eq9},\\c_{3}-d_{3}&=\frac{c_{1}^{3}}{2}+\frac{(4-c_{1}^{2})c_{1}}{2}(x+y)-\frac{(4-c_{1}^{2})c_{1}}{4}(x^{2}+y^{2})+\frac{4-c_{1}^{2}}{2}\left[ (1-|x|^{2})z-(1-|y|^{2})w\right]\label{eq10}, 
\end{align}
for some $x$, $y$ and $z$, $w$ with $|x|\leq1$, $|y|\leq1$, $|z|\leq1$ and $|w|\leq1$. Next, we will need to plug the last three equations given by (\ref{eq8}), (\ref{eq9}) and (\ref{eq10}) into the Hankel functional given by the equation (\ref{eq7}). And also, with the help of an application of  the triangle inequality, we have
\begin{align*}
\left| a_{2}a_{4}-a_{3}^{2}\right|&\leq \frac{U_{1}(t)\left|\Omega_{\lambda,\mu,\delta}(t) \right| }{96\delta(\delta^{3}+2\delta)(2\delta^{2}+1)^{2}(3\lambda+\mu)(\lambda+\mu)^{4}}c_{1}^{4}+\frac{3U_{1}^{2}(t)c_{1}(4-c_{1}^{2})}{8\delta(\delta^{3}+2\delta)(\lambda+\mu)(3\lambda+\mu)}\nonumber \\
 &+\bigg[\frac{\left(15(2\delta^{2}+1)^{2}-36\delta(\delta^{3}+2\delta) \right)U_{1}^{3} }{64\delta(\delta^{3}+2\delta)(2\delta^{2}+1)^{2}(\lambda+\mu)^{2}(2\lambda+\mu)}+\frac{3U_{1}(t) U_{2}(t) }{16\delta(\delta^{3}+2\delta)(\lambda+\mu)(3\lambda+\mu)}\bigg]c_{1}^{2}(4-c_{1}^{2})(|x|+|y|)\nonumber\\&+\bigg[\frac{3U_{1}^{2}(t)c_{1}^{2}(4-c_{1}^{2})}{32\delta(\delta^{3}+2\delta)(\lambda+\mu)(3\lambda+\mu)}-\frac{3U_{1}^{2}(t)c_{1}(4-c_{1}^{2})}{16\delta(\delta^{3}+2\delta)(\lambda+\mu)(3\lambda+\mu)}\bigg](|x|^{2}+|y|^{2})\nonumber \\&+\frac{9U_{1}^{2}(t)(4-c_{1}^{2})^{2}}{64(2\delta^{2}+1)^{2}(2\lambda+\mu)^{2}}(|x|+|y|)^{2},
\end{align*}
where
$$\Omega_{\lambda,\mu,\delta}(t)=18(2\delta^{2}+1)^{2}(\lambda+\mu)^{3}U_{3}(t)-U_{1}^{3}(t)(3\lambda+\mu)\bigg(3(2\delta^{2}+1)^{2}(\mu^{2}+3\mu-4)+54\delta(\delta^{3}+2\delta)\bigg).$$
Since the class $\mathcal{P}$ is invaryant under the rotations, by (\ref{lem1ineq}) we may suppose without loss of generality that $c_{1}:=c\in[0,2]$. Therefore, for $\eta=|x|\leq1$ and $\zeta=|y|\leq1$, we obtain
$$\left|a_{2}a_{4}-a_{3}^{2}\right|\leq\kappa_{1}+\kappa_{2}(\gamma_{1}+\gamma_{2})+\kappa_{3}(\gamma_{1}^{2}+\gamma_{2}^{2})+\kappa_{4}(\gamma_{1}+\gamma_{2})^{2}=\psi(\gamma_{1},\gamma_{2}) $$
where
\begin{align}
\kappa_{1}&=\frac{U_{1}(t)\left|\Omega_{\lambda,\mu,\delta}(t) \right| }{96\delta(\delta^{3}+2\delta)(2\delta^{2}+1)^{2}(3\lambda+\mu)(\lambda+\mu)^{4}}c^{4}+\frac{3U_{1}^{2}(t)c(4-c^{2})}{8\delta(\delta^{3}+2\delta)(\lambda+\mu)(3\lambda+\mu)}\geq 0 \nonumber\\ \kappa_{2}&=\bigg[\frac{\left(15(2\delta^{2}+1)^{2}-36\delta(\delta^{3}+2\delta) \right)U_{1}^{3} }{64\delta(\delta^{3}+2\delta)(2\delta^{2}+1)^{2}(\lambda+\mu)^{2}(2\lambda+\mu)}+\frac{3U_{1}(t)U_{2}(t) }{16\delta(\delta^{3}+2\delta)(\lambda+\mu)(3\lambda+\mu)}\bigg]c^{2}(4-c^{2})\geq 0 \nonumber
\\\kappa_{3}&=\frac{3U_{1}^{2}(t)c(c-2)(4-c^{2})}{32\delta(\delta^{3}+2\delta)(\lambda+\mu)(3\lambda+\mu)}\leq 0 \nonumber \\ \nonumber  \kappa_{4}&=\frac{9U_{1}^{2}(t)(4-c^{2})^{2}}{64(2\delta^{2}+1)^{2}(2\lambda+\mu)^{2}}\geq 0, \text{ \ \ } \frac{1}{2}<t<1.
\end{align}
Now, let us consider the closed square $\mathbb{S}=\left\lbrace (\gamma_{1},\gamma_{2}):0\leq \gamma_{1} \leq 1, 0\leq \gamma_{2}\leq1\right\rbrace $. In that case, all that we need to do is to maximize the function $\psi(\gamma_{1},\gamma_{2})$ in the closed square $\mathbb{S}$ for $c\in[0,2]$. Since $\kappa_{3}\leq0$ and $\kappa_{3}+2\kappa_{4}\geq0$ for all $t\in(\frac{1}{2},1)$ and $c\in(0,2)$, we conclude that
$$\psi_{\gamma_{1}\gamma_{1}}\psi_{\gamma_{2}\gamma_{2}}-(\psi_{\gamma_{1}\gamma_{2}})^{2}<0, \text{ \ for all \ } \gamma_{1},\gamma_{2}\in\mathbb{S}.$$
Thus, the function $\psi$ cannot have a local maximum in the interior of the square $\mathbb{S}$. Now, we investigate the maximum of $\psi$ on the boundary of the square $\mathbb{S}$.

For $\gamma_{1}=0$ and $0\leq\gamma_{2}\leq1$ (similarly $\gamma_{2}=0$ and $0\leq\gamma_{1}\leq1$) we get
$$\psi(0,\gamma_{2})=\phi(\gamma_{2})=\kappa_{1}+\kappa_{2}\gamma_{2}+(\kappa_{3}+\kappa_{4})\gamma_{2}^{2}.$$

Next, we are going to be dealing with the following two cases separately.

\begin{description}
\item[Case 1] Let $\kappa_{3}+\kappa_{4}\geq 0$. In this case for $0<\gamma_{2}<1$, any fixed $c$ with $0\leq c <2$ and for all $t$ with $\frac{1}{2}<t<1$, it is clear that $\phi^{\prime}(\gamma_{2})=2(\kappa_{3}+\kappa_{4})\gamma_{2}+\kappa_{2}>0$, that is, $\phi(\gamma_{2})$ is an increasing function. Hence, for fixed $c\in[0,2)$ and $t\in(1/2,1)$, the function $\phi(\gamma_{2})$ attains a maximum at $\gamma_{2}=1$ and
$$\max\left( \phi(\gamma_{2})\right) =\phi(1)=\kappa_{1}+\kappa_{2}+\kappa_{3}+\kappa_{4}.$$

\item[Case 2] Let $\kappa_{3}+\kappa_{4}<0$. Since $\kappa_{2}+2(\kappa_{3}+\kappa_{4})\geq0$ for $0<\gamma_{2}<1$, any fixed $c$ with $0\leq c<2$ and for all $t\in(1/2,1)$, it is clear that $\kappa_{2}+2(\kappa_{3}+\kappa_{4})<2(\kappa_{3}+\kappa_{4})\gamma_{2}+\kappa_{2}<\kappa_{2}$ and so $\phi^{\prime}(\gamma_{2})>0$. Hence, for fixed $c\in[0,2)$ and $t\in(\frac{1}{2},1)$, the function $\phi(\gamma_{2})$ attains a maximum at $\gamma_{2}=1$.
\end{description}
For $\gamma_{1}=1$ and $0\leq\gamma_{2}\leq1$ (similarly $\gamma_{2}=1$ and $0\leq\gamma_{1}\leq1$), we get
$$\psi(1,\gamma_{2})=\Psi(\gamma_{2})=(\kappa_{3}+\kappa_{4})\gamma_{2}^{2}+(\kappa_{2}+2\kappa_{4})\gamma_{2}+\kappa_{1}+\kappa_{2}+\kappa_{3}+\kappa_{4}.$$
Thus, from the above cases of $\kappa_{3}+\kappa_{4}$, we get that
$$\max\Psi(\gamma_{2})=\Psi(1)=\kappa_{1}+2\kappa_{2}+2\kappa_{3}+4\kappa_{4}.$$
Since $\phi(1)\leq\Psi(1)$ for $c\in(0,2)$ and $t\in(\frac{1}{2},1)$, we obtain
$$\max(\psi(\gamma_{1},\gamma_{2}))=\psi(1,1)$$
on the boundary of the square $\mathbb{S}$. Thus, the maximum of $\psi$ occurs at $\gamma_{1}=1$ and $\gamma_{2}=1$ in the closed square $\mathbb{S}$.

Let a function $\varphi:[0,2]\rightarrow \mathbb{R}$ defined by
\begin{equation}\label{equ19}
\varphi(c,t)=\max\left(\psi(\gamma_{1},\gamma_{2}) \right)=\psi(1,1)=\kappa_{1}+2\kappa_{2}+2\kappa_{3}+4\kappa_{4} 
\end{equation}
for fixed value of $t$. Substituting the values of $\kappa_{1}, \kappa_{2}, \kappa_{3}$ and $\kappa_{4}$ in the function $\varphi$ defined by (\ref{equ19}), yields
$$\varphi(c,t)=\frac{9U_{1}^{2}(t)}{(2\delta^{2}+1)^{2}(2\lambda+\mu)^{2}}+\frac{\chi_{1}c^{4}+36\chi_{2}c^{2}}{96\delta(\delta^{3}+2\delta)(2\delta^{2}+1)^{2}(3\lambda+\mu)(2\lambda+\mu)^{2}(\lambda+\mu)^{4}},$$
where
\begin{align*}
\chi_{1}&=(2\lambda+\mu)^{2}U_{1}(t)\left|\Omega_{\lambda,\mu,\delta}(t) \right|+18(\lambda+\mu)^{3}\bigg( 3\delta(\delta^{3}+2\delta)(\lambda+\mu)(3\lambda+\mu)-(2\delta^{2}+1)^{2}(2\lambda+\mu)^{2}\bigg)U_{1}^{2}(t)
\\&-9(\lambda+\mu)^{2}(2\lambda+\mu)U_{1}(t)\bigg((3\lambda+\mu)(8\delta^{4}-4\delta^{2}+5)U_{1}^{2}(t)+4(2\delta^{2}+1)^{2}(2\lambda+\mu)(\lambda+\mu)U_{2}(t)\bigg),
\\\chi_{2}&=\bigg[(2\lambda+\mu)(3\lambda+\mu)\left(8\delta^{4}-4\delta^{2}+5\right)U_{1}^{3}(t)+4(2\delta^{2}+1)^{2}(\lambda+\mu)(2\lambda+\mu)^{2}U_{1}(t)U_{2}(t)
\\&+(\lambda+\mu)U_{1}^{2}(t)\big(2(2\delta^{2}+1)^{2}(2\lambda+\mu)^{2}-12\delta(\delta^{3}+2\delta)(\lambda+\mu)(3\lambda+\mu)\big) \bigg](\lambda+\mu)^{2}. 
\end{align*}
Assume that $\varphi(c,t)$ has a maximum value in an interior of $c\in[0,2]$, by elementary calculation, we obtain that
$$\varphi^{\prime}(c,t)=\frac{\chi_{1}c^{3}+18\chi_{2}c}{24\delta(\delta^{3}+2\delta)(2\delta^{2}+1)^{2}(3\lambda+\mu)(2\lambda+\mu)^{2}(\lambda+\mu)^{4}}.$$
We will examine the sign of the function $\varphi^{\prime}(c,t)$ depending on the different cases of the signs of $\chi_{1}$ and $\chi_{2}$ as follows:
\begin{enumerate}
\item Let $\chi_{1}\geq0$ and $\chi_{2}\geq0$, then  $\varphi^{\prime}(c,t)\geq0$, so  $\varphi(c,t)$ is an increasing function. Therefore
\begin{align}
\max\left\lbrace \varphi(c,t):c\in(0,2)\right\rbrace &=\varphi(2^{-},t)
\nonumber\\&=\frac{9U_{1}^{2}(t)}{(2\delta^{2}+1)^{2}(2\lambda+\mu)^{2}}+\frac{\chi_{1}+9\chi_{2}}{6\delta(\delta^{3}+2\delta)(2\delta^{2}+1)^{2}(3\lambda+\mu)(2\lambda+\mu)^{2}(\lambda+\mu)^{4}}.\label{equ20}
\end{align}
That is, $\max\left\lbrace \max\left\lbrace \psi(\gamma_{1},\gamma_{2}):0\leq\gamma_{1},\gamma_{2}\leq1\right\rbrace :0<c<2 \right\rbrace=\varphi(2^{-},t)$.
\item Let $\chi_{1}\leq0$ and $\chi_{2}\leq0$, then $\varphi^{\prime}(c,t)\leq0$, so $\varphi(c,t)$ is an decreasing function on the interval $(0,2)$. Therefore
\begin{equation}\label{equ21}
\max\left\lbrace \varphi(c,t):c\in(0,2) \right\rbrace=\varphi(0^{+},t)=4\kappa_{4}=\frac{9U_{1}^{2}(t)}{(2\delta^{2}+1)^{2}(2\lambda+\mu)^{2}}. 
\end{equation}
\item Let $\chi_{1}>0$ and $\chi_{2}<0$, then $c_{0}=\sqrt{\frac{-18\chi_{2}}{\chi_{1}}}$ is a critical point of the function $\varphi(c,t)$. We suppose that, $c_{0}\in(0,2)$, since $\varphi^{\prime \prime}(c_{0},t)>0$, $c_{0}$ is a local minimum point of the function $\varphi(c,t)$. That is the function $\varphi(c,t)$ cannot have a local maximum.
\item Let $\chi_{1}<0$ and $\chi_{2}>0$, then $c_{0}$ is a critical point of the function $\varphi(c,t)$. We assume that $c_{0}\in(0,2)$. Since $\varphi^{\prime \prime}(c_{0},t)<0$, $c_{0}$ is a local maximum point of the function $\varphi(c,t)$ and maximum value occurs at $c=c_{0}$. Therefore
\begin{equation}\label{equ22}
\max\left\lbrace\varphi(c,t):c\in(0,2) \right\rbrace=\varphi(c_{0},t) 
\end{equation}
where
$$\varphi(c_{0},t)=\frac{9U_{1}^{2}(t)}{(2\delta^{2}+1)^{2}(2\lambda+\mu)^{2}}-\frac{27\chi_{2}^{2}}{8\chi_{1}\delta(\delta^{3}+2\delta)(2\delta^{2}+1)^{2}(3\lambda+\mu)(2\lambda+\mu)^{2}(\lambda+\mu)^{4}}.$$
\end{enumerate}
Thus, from (\ref{equ20}) to (\ref{equ22}), the proof of Theorem \ref*{theo1} is completed.
\end{proof}

Now, we would like to draw attention to some remarkable results which are obtained for some values of $\lambda$, $\mu$ and $\delta$ in Theorem \ref{theo1}.
\begin{corollary}\label{corollary1}
Let $h_{\delta}\in \sigma$ of the form (\ref{functionh}) be in $\mathcal{B}_{\sigma}^{\delta}(\lambda,t)$. Then
\begin{align*}
\left|a_{2}a_{4}-a_{3} ^{2}\right| &\leq
\begin{cases}
\varphi(2^{-},t)      ,&\chi_{3}\geq0 \text{ \and\ } \chi_{4}\geq0\\
\frac{36t^{2}}{(2\delta^{2}+1)^{2}(2\lambda+1)^{2}}, & \chi_{3}\leq0 \text{ \and\ } \chi_{4}\leq0\\
\max\left\lbrace\frac{36t^{2}}{(2\delta^{2}+1)^{2}(2\lambda+1)^{2}},\text{ \ \ }\varphi(2^{-},t) \right\rbrace, &\chi_{3}>0 \text{ \and\ } \chi_{4}<0\\
\max\left\lbrace\varphi(c_{0},t), \varphi(2^{-},t) \right\rbrace  ,&\chi_{3}<0 \text{ \and\ } \chi_{4}>0
\end{cases},
\end{align*}
where
$$\varphi(2^{-},t)=\frac{9U_{1}^{2}(t)}{(2\delta^{2}+1)^{2}(2\lambda+1)^{2}}+\frac{\chi_{3}+9\chi_{4}}{6\delta(\delta^{3}+2\delta)(2\delta^{2}+1)^{2}(3\lambda+1)(2\lambda+1)^{2}(\lambda+1)^{4}},$$
$$\varphi(c_{0},t)=\frac{9U_{1}^{2}(t)}{(2\delta^{2}+1)^{2}(2\lambda+1)^{2}}-\frac{27\chi_{4}^{2}}{8\chi_{3}\delta(\delta^{3}+2\delta)(2\delta^{2}+1)^{2}(3\lambda+1)(2\lambda+1)^{2}(\lambda+1)^{4}}, \text{ \ \ \ } c_{0}=\sqrt{\frac{-18\chi_{4}}{\chi_{3}}}$$
and $\chi_{3}=\chi_{1}(\lambda,\mu=1,\delta;t)$, $\chi_{4}=\chi_{2}(\lambda,\mu=1,\delta;t)$.
\end{corollary}
Taking $\lambda=1$ in the Theorem \ref{theo1}, we get the following result.
\begin{corollary}\label{corollary2}
Let $h_{\delta}\in \sigma$ of the form (\ref{functionh}) be in $\mathcal{B}_{\sigma}^{\mu,\delta}(t)$. Then
\begin{align*}
\left|a_{2}a_{4}-a_{3} ^{2}\right| &\leq
\begin{cases}
\varphi(2^{-},t),      &\chi_{5}\geq0 \text{ \and\ } \chi_{6}\geq0\\
\frac{36t^{2}}{(2\delta^{2}+1)^{2}(2+\mu)^{2}}, &\chi_{5}\leq0 \text{ \and\ } \chi_{6}\leq0\\
\max\left\lbrace\frac{36t^{2}}{(2\delta^{2}+1)^{2}(2+\mu)^{2}},\text{ \ \ }\varphi(2^{-},t) \right\rbrace ,&\chi_{5}>0 \text{ \and\ } \chi_{6}<0\\
\max\left\lbrace\varphi(c_{0},t), \varphi(2^{-},t) \right\rbrace,  &\chi_{5}<0 \text{ \and\ } \chi_{6}>0
\end{cases},
\end{align*}
where
$$\varphi(2^{-},t)=\frac{9U_{1}^{2}(t)}{(2\delta^{2}+1)^{2}(2+\mu)^{2}}+\frac{\chi_{5}+9\chi_{6}}{6\delta(\delta^{3}+2\delta)(2\delta^{2}+1)^{2}(3+\mu)(2+\mu)^{2}(1+\mu)^{4}},$$
$$\varphi(c_{0},t)=\frac{9U_{1}^{2}(t)}{(2\delta^{2}+1)^{2}(2+\mu)^{2}}-\frac{27\chi_{6}^{2}}{8\chi_{5}\delta(\delta^{3}+2\delta)(2\delta^{2}+1)^{2}(3+\mu)(2+\mu)^{2}(1+\mu)^{4}}, \text{ \ \ \ } c_{0}=\sqrt{\frac{-18\chi_{6}}{\chi_{5}}}$$
and $\chi_{5}=\chi_{1}(\lambda=1,\mu,\delta;t)$, $\chi_{6}=\chi_{2}(\lambda=1,\mu,\delta;t)$.
\end{corollary}
Putting $\lambda=1$ and $\mu=1$ in the Theorem \ref{theo1}, we find the following result.
\begin{corollary}\label{corollary3}
If $h_{\delta}\in \mathcal{B}_{\sigma}^{\delta}(t)$ is of the form (\ref{functionh}). Then
\begin{align*}
\left| a_{2}a_{4}-a_{3}^{2}\right|\leq
\begin{cases}
\varphi(2^{-},t), &\frac{1}{2}<t\leq t_{0}\\
\varphi(c_{0},t), &t_{0}< t <1
\end{cases},
\end{align*}
where
$$\varphi(2^{-},t)=\frac{3t^{2}\big((2\delta^{2}+1)^{2}-t(5\delta^{4}+2\delta^{2}+2)\big)}{(\delta^{2}+2)(2\delta^{3}+\delta)^{2}},$$
$$\varphi(c_{0},t)=\frac{4t^{2}}{(2\delta^{2}+1)^{2}}-\frac{t(\nabla_{(4,20,-3);(22,40,64)}^{(3,1,-1)})^{2}}{\delta(2\delta^{2}+1)^{2}(\delta^{3}+2\delta)\big(\nabla_{(-4,4,-3);(22,40,64)}^{(3,1,-1)}+6t\left|t(5\delta^{4}+2\delta^{2}+2)-(2\delta^{2}+1)^{2}
\right|\big)} ,$$
$\nabla_{(a,b,c);(d,e,f)}^{(m,n,r)}(\delta;t)=\nabla_{(a,b,c);(d,e,f)}^{(m,n,r)}=m(1+2\delta^{2})^{2}+nt(a+bt^{2}+ct^{4})+rt^{2}(d+et^{2}+ft^{4})$. Moreover, the value of $t_{0}$ is root of equation $\chi_{1}=0$ for $\lambda=\mu=1$ and $\frac{1}{2}<t<1$. 
\end{corollary}
Setting $\delta=1$ in the Corollary \ref{corollary3}, we obtain the following result.
\begin{corollary} 
If $h_{\delta}\in \mathcal{B}_{\sigma}(t)$ is of the form (\ref{functionh}). Then
\begin{align*}
\left| a_{2}a_{4}-a_{3}^{2}\right|\leq
\begin{cases}
t^{2}(1-t^{2}),  &\frac{1}{2}<t\leq t_{0_{1}}\\
\frac{t(260t^{4}+84t^{3}-139t^{2}-18t+9)}{8(18t^{3}+42t^{2}-17t-9)}  ,&t_{0_{1}}<t<1
\end{cases},
\end{align*}
where, the value of $t_{0_{1}}$, which is aproximately $t_{0_{1}}=0.603615$, is root of equation $\chi_{1}=0$ for $\lambda=\mu=\delta=1$ and $\frac{1}{2}<t<1$.
\end{corollary}

Next, taking $\lambda=1$ and $\mu=0$ in the Theorem \ref*{theo1}, we arrive at the following result.
\begin{corollary}\label{corollary5}
If $h_{\delta}\in \mathcal{S}_{\sigma}^{\star,\delta}(t)$ is of the form (\ref{functionh}), then
\begin{align*}
\left| a_{2}a_{4}-a_{3}^{2}\right|\leq
\begin{cases}
\varphi(2^{-},t),  &\frac{1}{2}<t\leq t_{0_{2}}\\
\varphi(c_{0},t), &t_{0_{2}}<t<1
\end{cases},
\end{align*}
where,
$$\varphi(2^{-},t)=\frac{8t^{2}\big((2\delta^{2}+1)^{2}-6t^{2}(\delta^{2}-1)^{2}\big)}{(2\delta^{3}+\delta)^{2}(\delta^{2}+2)},$$
$$\varphi(c_{0},t)=\frac{9t^{2}}{(2\delta^{2}+1)^{2}}-\frac{t(\nabla_{(-2,10,1);(23,20,56)}^{(-2,-1,1)})^{2}}{\delta(2\delta^{2}+1)^{2}(\delta^{3}+2\delta)\big(\nabla_{(4,-2,7);(23,20,56)}^{(-4,1,2)}-8t\left|6t^{2}(\delta^{2}-1)^{2}-(2\delta^{2}+1)^{2}\right|\big)}.$$
Moreover, the value of $t_{0_{2}}$ is root of equation  $\chi_{1}=0$ for $\lambda=1$, $\mu=0$ and $\frac{1}{2}<t<1$. 
\end{corollary}
Now, taking $\delta=1$ in the Corollary \ref{corollary5}, we attain the following result.
\begin{corollary}
If $h_{\delta}\in \mathcal{S}_{\sigma}^{\star}(t)$ is of the form (\ref{functionh}), then
\begin{align*}
\left| a_{2}a_{4}-a_{3}^{2}\right|\leq
\begin{cases}
\frac{8t^{2}}{3}, & \frac{1}{2}<t\leq\frac{7+\sqrt{401}}{44}\\
t^{2}+\frac{t(2+t-11t^{2})^{2}}{3(-4-7t+22t^{2})}, &\frac{7+\sqrt{401}}{44}<t<1
\end{cases}.
\end{align*}
\end{corollary}
When $\delta=1$, we note that the results given the above coincide with the results in \cite{omb1}.


\begin{thebibliography}{width}
\bibitem{alrs}
\textsc{R. M. Ali, S. K. Lee, V. Ravichandran, S. Supramanian},
The Fekete-Szeg\"{o} coefficient functional for transforms of analytic functions, {\em Bull. Iranian Math. Soc.}, 35 (2009), no. 2, 119-142, 276.

\bibitem{alr}
\textsc{R. M. Ali, S. K. Lee, V. Ravichandran, S. Supramanian},
Coefficient estimates for bi-univalent Ma-Minda starlike and convex functions, {\em Appl. Math. Lett.} 25 (3) (2012), 344-351.

\bibitem{ay0}
\textsc{\c{S}. Alt{\i}nkaya, S. Yal\c{c}{\i}n}, Chebyshev polynomial coefficient bounds for a subclass of bi-univalent functions, arXiv: 1605.08224v1.

\bibitem{ay1}
\textsc{\c{S}. Alt{\i}nkaya, S. Yal\c{c}{\i}n}, Upper bound second Hankel determinant for bi-Bazilevi\u{c}, {\em Mediterr. J. Math.} 13 (2016), no. 6, 4081-4090.

\bibitem{ay2}
\textsc{\c{S}. Alt{\i}nkaya, S. Yal\c{c}{\i}n}, Construction of second Hankel determinant for a new subclass bi-univalent functions, {\em Turkish. J. Math.}; doi: 10.3906/mat-1507-39.

\bibitem{ay3}
\textsc{\c{S}. Alt{\i}nkaya, S. Yal\c{c}{\i}n}, On the Chebyshev polynomial coefficient problem of bi-Bazilevi\u{c} function, {\em Commun. Fac. Sci. Univ. Ank. Series A1}; doi: 10.1501/Commua1-0000000819.

\bibitem{bma}
\textsc{S. Bulut, N. Magesh, C. Abirami}, A comprehensive class of analytic bi-univalent functions by means of Chebyshev polynomials, {\em J. Fract. Calc. Appl.} 8 (2017), no 2, 32-39.

\bibitem{bmb}
\textsc{S. Bulut, N. Magesh, V.K. Balaji}, Initial bounds for analytic and bi-univalent functions by means of Chebyshev polynomials, {\em J. Class. Anal.} 8 (2017), In-press.

\bibitem{cantor}
\textsc{D.G. Cantor},  Power series with the integral coefficients, {\em Bull. Amer. Math. Soc.}, 69(1963), 362-366.

\bibitem{chihara}
\textsc{T.S. Chihara}, {\em An Introduction to Orthogonal Polynomials}, Gordon and Breach, Science Publisher Inc., New York, (1978).

\bibitem{coy}
\textsc{M. \c{C}a\u{g}lar, H. Orhan, N. Ya\u{g}mur}, Coefficient bounds for new subclasses of bi-univalent functions, {\em Filomat} 27 (2013), no 7, 1165-1171.

\bibitem{dt}
\textsc{V.K. Deekonda, R. Thoutreedy}, An upper bound to the second Hankel determinant for functions in Mocanu class, {\em Vietnam J. Math.}, (2015) 43:541-549.

\bibitem{dco}
\textsc{E. Deniz, M. \c{C}a\u{g}lar, H. Orhan}, Second Hankel determinant for bi-starlike and bi-convex functions of oder $\beta$, {\em Appl. Math. Comput.} 271 (2015), 301-307.

\bibitem{dienes} 
\textsc{P. Dienes}, {\em  The Taylor series: An introduction to the theory of functions of a complex vartiable}, Dover, New York, 1957. 

\bibitem{doha}
\textsc{E.H. Doha}, The first and second kind Chebyshev coefficients of the moments of the general-order derivative of an infinitely differentiable function, {\em Int. J. Comput. Math.} 51 (1994), 23-35.

\bibitem{drs}
\textsc{J. Dziok, R.K. Raina, J. Sok\'{o}l}, Applications of Chebyshev polynomials to classes of analytic functions, {\em C.R. Math. Acad. Sci. Paris}, 353(5) (2015), 433-438.

\bibitem{gs}
\textsc{U. Grenander, G. Szeg\"{o}}, {\em  Toeplitz Forms and Their Applications},California Monographs in Mathematical Sciences. Berkeley, CA, USA: University California Press, 1958.

\bibitem{keme} 
\textsc{F.R. Keogh, E.P. Merkes}, A coefficient inequality for certain classes of analytic functions, {\em Proc. Amer. Math. Soc.} 20 (1969), 8-12.

\bibitem{lrs}
\textsc{K. Lee, V. Ravichandran, S. Supramaniam}, Bounds for the second Hankel determinant of certain univalent functions, {\em J. Inequal. Appl.}, 281 (2013), 1-17.

\bibitem{mustafa}
\textsc{N. Mustafa}, Upper bound for the second Hankel determinant of certain subclass of analytic and bi-univalent functions, {\em arXiv-1702.06826}.

\bibitem{nt}
\textsc{J.W. Noonan, D.K. Thomas}, On the second Hankel determinant of a really mean p-valent functions, {\em Trans. Amer. Math. Soc.}  (1976) 223: 337-346.

\bibitem{omb}
\textsc{H. Orhan, N. Magesh, V.K. Balaji}, Fekete-Szeg\"{o} problem for certain classes of Ma-Minda bi-univalent functions, {\em Afr. Math.}, (2016) 27: 889-897.

\bibitem{omb1}
\textsc{H. Orhan, N. Magesh, V.K. Balaji}, Second Hankel determinant for certain class of bi-univalent functions defined by Chebyshev polynomials, {\em 	arXiv:1705.03313.}

\bibitem{omy}
\textsc{H. Orhan, N. Magesh, J. Yamini}, Bounds for the second Hankel determinant of certain bi-univalent functions, {\em Turkish J. Math.}, 40 (2016), no. 3, 679-687.

\bibitem{otk}
\textsc{H. Orhan, E. Toklu, E. Kad{\i}o{\u{g}}lu}, Second Hankel determinant problem for k-bi-starlike functions, {\em Filomat}, (accepted).

\bibitem{pommerenke}
\textsc{C. Pommerenke}, {\em Univalent Functions}. Gottingen, Germany: Vandenhoeck and Rupercht, 1975.
	
\bibitem{smg}
\textsc{H. M. Srivastava, A.K. Mishra, P. Gochhayat}, Certain subclasses of analytic and bi-univalent functions, {\em Appl. Math. Lett.} 23 (10) (2010), 1188-1192.

\bibitem{trimble}
\textsc{S.Y. Trimble}, A coefficient inequality for convex univalent functions, {\em Proc. Amer. Math. Soc.} Volume 48, no. 1 (1975).

\bibitem{ww}
\textsc{T. Whittaker, G.N. Watson}, {\em A course of modern analysis}, reprint of the fourth (1927) edition, Cambridge Mathematical Library, Cambridge Univ. Press, Cambvridge, 1996.	

	
\end{thebibliography}
\end{document}